\documentclass[english]{amsart}

\title{On the number of cusps on cuspidal curves on Hirzebruch surfaces}
\author{Torgunn Karoline Moe}
\address{Department of Mathematics, University of Oslo, P.O. Box 1053 Blindern, NO-0316 Oslo, NORWAY}
\email{torgunnk@math.uio.no}



\date{\today}

\usepackage{babel}
\usepackage[latin1]{inputenc}
\usepackage{mathrsfs}
\usepackage{graphicx}
\usepackage{hyperref}
\usepackage{amssymb,amsfonts,amssymb,amsthm}
\usepackage{moreverb}
\usepackage{subfigure}
\usepackage{multirow}
\usepackage{array,float}
\usepackage{enumerate}

\usepackage{booktabs}

\input xy
\xyoption{all}

\theoremstyle{definition}

\theoremstyle{plain}
\newtheorem{thm}{Theorem}[section]
\newtheorem{lem}[thm]{Lemma}
\newtheorem{prp}[thm]{Proposition}
\newtheorem{cor}[thm]{Corollary}
\newtheorem*{cor*}{Corollary}
\newtheorem{conj}[thm]{Conjecture}

\newtheorem*{conj*}{Conjecture}
\newtheorem*{thm*}{Theorem}
\newtheorem*{prb*}{Problem}
\newtheorem*{qe*}{Question}

\theoremstyle{definition}

\theoremstyle{remark}

\numberwithin{equation}{section}

\newcommand{\Po}{\ensuremath{\mathbb{P}^2}}

\newcommand{\Fe}{\ensuremath{\mathbb{F}_e}}

\newcommand{\Fen}{\ensuremath{\mathbb{F}_1}}
\newcommand{\Pot}{\ensuremath{\mathbb{P}^1}}

\newcommand{\C}{\ensuremath{\mathbb{C}}}
\renewcommand{\H}{\mathrm{H}}
\newcommand{\h}{\mathrm{h}}
\newcommand{\Z}{\ensuremath{\mathbb{Z}}}
\newcommand{\N}{\ensuremath{\mathbb{N}}}
\newcommand{\Q}{\ensuremath{\mathbb{Q}}}

\newcommand{\V}{\ensuremath{\mathscr{V}}}
\newcommand{\D}{\ensuremath{\mathscr{D}}}
\newcommand{\RH}{M}
\newcommand{\LH}{L}
\newcommand{\qs}{\ensuremath{q}}
\newcommand{\Os}{\ensuremath{\mathscr{O}}}
\newcommand{\noi}{\noindent}
\newcommand{\ol}{\overline}

\newcommand{\beq}{\begin{equation*}}
\newcommand{\eeq}{\end{equation*}}
\newcommand{\bsp}{\begin{split}}
\newcommand{\esp}{\end{split}}

\numberwithin{equation}{section}

\newcommand{\lkkf}{\ol{\kappa}(\Fe \setminus C)}

\newcommand{\ldot}{\ensuremath{\,.\,}}
\newcommand{\hide}[1]{}
\begin{document}

\begin{abstract}
In this article we give an upper bound for the number of cusps on a cuspidal curve on a Hirzebruch surface. We adapt the results that have been found for a similar question asked for cuspidal curves on the projective plane, and restate the results in this new setting. 
\end{abstract}

\maketitle

\setcounter{tocdepth}{1}

\tableofcontents

\section{Introduction}
Let $C$ be a reduced and irreducible curve of geometric genus $g$ on a smooth complex surface $X$. A point $p$ on $C$ is called a \emph{cusp} if it is singular and if the germ $(C,p)$ of $C$ at $p$ is irreducible. A curve $C$ is called \emph{cuspidal} if all of its singularities are cusps. 

For any two divisors $C$ and $C'$ on $X$, we calculate the \emph{intersection number} $C \ldot C'$ using linear equivalence and the pairing ${\mathrm{Pic}(X) \times \mathrm{Pic}(X) \rightarrow \Z}$ \cite[Theorem V 1.1, pp.357--358]{Hart:1977}.

By \cite[Theorem V 3.9, p.391]{Hart:1977}, there exists for any curve $C$ on a surface $X$ a sequence of $t$ monoidal transformations, $$V=V_t \xrightarrow{\sigma_t} V_{t-1} \xrightarrow{} \cdots \xrightarrow{} V_1 \xrightarrow{\sigma_1} V_0=X,$$ such that the reduced total inverse image of $C$ under the composition $\sigma:V \rightarrow X$, $$D := \sigma^{-1}(C)_{\mathrm{red}},$$ is a \emph{simple normal crossing divisor} (SNC-divisor) on the smooth complete surface $V$ (see \cite{Iitaka}). The pair $(V,D)$ and the transformation $\sigma$ are referred to as an \emph{embedded resolution} of $C$, and it is called a \emph{minimal embedded resolution} of $C$ when $t$ is the smallest integer such that $D$ is an SNC-divisor.

Let $p$ be a cusp on a curve $C$, let $m$ denote the multiplicity of $p$, and let $m_i$ denote the multiplicity of the infinitely near points $p_i$ of $p$. Then the \emph{multiplicity sequence} $\ol{m}$ of the cusp $p$ is defined to be the sequence of integers $$\ol{m}=[m,m_1,\ldots,m_{t-1}],$$ where $t$ is the number of monoidal tranformations in the local minimal embedded resolution of the cusp, and we have $m_{t-1}=1$ (see \cite{Brieskorn}). \hide{We follow the convention of compacting the notation by omitting the number of ending 1s in the sequence and indexing the number of repeated elements, for example, we write $$[6,6,3,3,3,2,1,1]=[6_2,3_3,2].$$

The collection of multiplicity sequences of a cuspidal curve will be referred to as its \emph{cuspidal configuration}.}

Now the question of how many cusps a cuspidal curve on a surface can have naturally arises. The main result in this article is an upper bound for the number of cusps on cuspidal curves on Hirzebruch surfaces. Note that the theorem and the proof is very similar to the proof in the case of plane curves (cf. \cite{Tono05}).

\begin{thm*}\label{TTH}
The number of cusps $s$ on a cuspidal curve $C$ of genus $g$ on a Hirzebruch surface has an upper bound, $$s \leq \frac{21g+29}{2}.$$
\end{thm*}

In particular, we have the following corollary.
\begin{cor*}
A rational cuspidal curve on $\Fe$ can not have more than 14 cusps.
\end{cor*}

\subsection{Structure}
In Section \ref{M} we motivate the study of the question of how many cusps a cuspidal curve on a Hirzebruch surface can have by recalling the history of the study of this problem on the projective plane. In Section \ref{PR} we give the basic definitions and preliminary results for cuspidal curves on Hirzebruch surfaces. Section \ref{ON} contains the main result of this article. Here we state and prove the above theorem. 

\subsection{Acknowledgements}
This article consists of results from my PhD-thesis \cite{MOEPHD}, and it is the first of two articles (see \cite{MOECCH}). I am very grateful to Professor Ragni Piene for suggesting cuspidal curves on Hirzebruch surfaces as the topic of my thesis, and for all the help along the way. Moreover, I am indebted to Georg Muntingh and to Nikolay Qviller for guiding me over some of the obstacles that I met in this work. Furthermore, I would like to thank Professor Keita Tono for explaining me important details, and Professor Hubert Flenner and Professor Mikhail Zaidenberg for valuable comments and suggestions.

\section{Motivation: the case of plane curves}\label{M}
Let $\Po$ denote the projective plane with coordinates $(x:y:z)$ and coordinate ring $\C[x,y,z]$. A reduced and irreducible curve $C$ on $\Po$ is given as the zero set $\V(F)$ of a homogeneous, reduced and irreducible polynomial $F(x,y,z) \in \C[x,y,z]_d$ for some $d$. In this case, the polynomial $F$ and the curve $C$ is said to have degree $d$.

Plane rational cuspidal curves have been studied quite intensively both classically and the last 20 years. Classically, the study was part of the process of classifying plane curves, and additionally bounds on the number of cusps were produced (see \cite{Clebsch, Lefschetz, Salmon, Telling, Veronese, Wieleitner}). In the modern context, rational cuspidal curves with many cusps play an important role in the study of open surfaces (see \cite{Fent, FlZa94, Wak}), and the study of these curves was further motivated in the mid 1990s by Sakai in \cite{Sakai}, when he suggested two open problems to be solved. The two tasks at hand were first to classify all rational and elliptic cuspidal plane curves, and second to find the maximal number of cusps on a rational cuspidal plane curve. In this article we only deal with the second problem for curves on Hirzebruch surfaces, and we consider the question for curves of any genus. 

The attempts to classify rational cuspidal curves has not been complete, but many curves have been found. In particular, a curve of degree 5 with four cusps was found, and the first mention of this curve that we have found is by Namba in \cite{Namba}. Moreover, three series of rational cuspidal curves with three cusps were constructed by Fenske in \cite{Fen99a}, and Flenner and Zaidenberg \cite{FlZa95,FlZa97}. The lack of examples of rational cuspidal curves with more than four cusps leads to a conjecture, originally proposed by Orevkov.

\begin{conj}
A plane rational cuspidal curve can not have more than four cusps.
\end{conj} 

Note that the conjecture is verified for rational cuspidal curves of degree $d \leq 20$ by Piontkowski in \cite{Piontkowski}, and that there is additional supporting evidence by Fenske in \cite{Fent}. 

Further attempts to prove the conjecture has not succeeded, but Tono published in \cite{Tono05} from 2005 a bound for the number of cusps on a plane cuspidal curve of genus $g$. Note that this bound depends only on the genus of the curve. The result says the following \cite[Theorem 1.1, p.216]{Tono05}.
\begin{thm}\label{ONCP2}
The number of cusps $s$ on a cuspidal curve $C$ of genus $g$ on $\Po$ has an upper bound, $$s \leq \frac{21g+17}{2}.$$
\end{thm}

The proof relies on properties of the dual graph of the minimal embedded resolution of a curve, and on the logarithmic Bogomolov--Miyaoka--Yau-inequality (B--M--Y-inequality). In the latter inequality, computing the logarithmic Kodaira dimension and the topologial Euler characteristic of the complement to the curve is essential.

Note that a similar bound, $s \leq 9$, was found for rational curves by Orevkov and Zaidenberg in \cite{Ore95}, but then under the assumption that the \emph{rigidity conjecture} of Flenner and Zaidenberg hold \cite{FlZa94, FlZa95}.

\section{Notation and preliminary results}\label{PR}
In this section we first recall general facts about curves on Hirzebruch surfaces. Second, we find the Euler characteristic of the complement to a curve, and use this to establish a logarithmic B--M--Y-inequality in this case. Last in this section we state and prove a result on the logarithmic Kodaira dimension of complements to curves on Hirzebruch surfaces that is similar to a result by Wakabayashi for complements of plane curves.

\subsection{A curve on a Hirzebruch surface}
Let $\mathbb{F}_e$ denote the Hirzebruch surface of type $e$ for any $e \geq 0$. Recall that $\Fe$ is a projective ruled surface, with $\Fe=\mathbb{P}(\Os \oplus \Os(-e))$ and morphism $\pi: \mathbb{F}_e \longrightarrow \Pot$. Moreover, $p_a(\mathbb{F}_e)=0$ and $p_g(\mathbb{F}_e)=0$ \cite[Corollary V 2.5, p.371]{Hart:1977}. The Hirzebruch surfaces are rational surfaces, relatively minimal in all cases except $e=1$. Indeed, the surface $\Fen$ is isomorphic to $\Po$ blown up in one point, and it contains an exceptional curve $E \cong \Pot$ with $E^2=-1$.

In the language of divisors, let $\LH$ be a \emph{fiber} of $\pi: \mathbb{F}_e \longrightarrow \Pot$ and $\RH_0$ the \emph{special section} of $\pi$. The Picard group of $\Fe$, $\mathrm{Pic}(\mathbb{F}_e)$, is isomorphic to $\Z \oplus \Z$. We choose $\LH$ and $\RH \sim e\LH+\RH_0$ as generators of ${\rm{Pic}}(\Fe)$, and we then have \cite[Theorem V 2.17, p.379]{Hart:1977} $$\LH^2=0, \qquad \LH \ldot \RH=1, \qquad \RH^2=e.$$ The canonical divisor $K$ on $\Fe$ can be expressed as \cite[Corollary V 2.11, p.374]{Hart:1977} $$K \sim (e-2)\LH-2\RH \; \text{ and }\; K^2=8.$$ Any irreducible curve $C\neq \LH,\RH_0$ corresponds to a divisor given by \cite[Proposition V 2.20, p.382]{Hart:1977} $$C \sim a\LH+b\RH,\quad b>0, \,a\geq 0.$$ The corresponding curve is said to be of type $(a,b)$.

For completion we include the genus formula for cuspidal curves on Hirzebruch surfaces. 
\begin{cor}[Genus formula]\label{genusfe}
A cuspidal curve $C$ of type $(a,b)$ with cusps $p_j$, for $j=1,\ldots,s$, and multiplicity sequences $\ol{m}_j=[m_0,m_1, \ldots, m_{t_j-1}]$ on the Hirzebruch surface $\mathbb{F}_e$ has genus $g$, where $$g=\frac{(b-1)(2a-2+be)}{2}-\sum_{j=1}^s \sum_{i=0}^{t_j-1} \frac{m_i(m_i-1)}{2}.$$
\end{cor}

\begin{proof}
Since $C \sim a\LH+b\RH$, $K \sim (e-2)\LH-2\RH$, $L^2=0$, $L \ldot M = 1$ and $M^2=e$, by the general genus formula \cite[Example V 3.9.2, p.393]{Hart:1977}, we have $$g=\frac{(a\LH+b\RH) \ldot (a\LH+b\RH+(e-2)\LH-2\RH)}{2}+1-\sum_{j=1}^s \delta_j,$$ where $\delta$ is the \emph{delta invariant}. This gives
\begin{align*}
g&=\frac{b^2e-2be+ab+be-2a+ab-2b}{2}+1-\sum_{j=1}^s \sum_{i=0}^{t_j-1} \frac{m_i(m_i-1)}{2}\\
&=\frac{(b-1)(2a-2+be)}{2}-\sum_{j=1}^s \sum_{i=0}^{t_j-1} \frac{m_i(m_i-1)}{2}.
\end{align*}
\end{proof}

\hide{
 and $\LH$ and $\RH_0$ can be chosen as generators of this group. We have \cite[Section V 2]{Hart:1977} $$\LH^2=0,\qquad \LH \ldot \RH_0=1, \qquad \RH_0^2=-e.$$ The canonical divisor $K$ can then be expressed as \cite[Corollary V 2.11, p.374]{Hart:1977} $$K \sim -(2+e)\LH -2\RH_0.$$ 

To simplify our calculations, we use another generating set of $\mathrm{Pic}(\Fe)$ \cite[Theorem V 2.17, p.379]{Hart:1977}. Let $\LH$ and $\RH \sim e\LH+\RH_0$ be the new generating set. Then we have $$\LH^2=0, \qquad \LH \ldot \RH=1, \qquad \RH^2=e.$$ Moreover, in this basis, $$K \sim (e-2)\LH-2\RH \; \text{ and }\; K^2=8.$$ Unless otherwise specified, we will use $\LH$ and $\RH$ as a basis for $\mathrm{Pic}(\Fe)$.

Any irreducible curve $C\neq \LH,\RH_0$ corresponds to a divisor on $\mathbb{F}_e$ given by \cite[Proposition V 2.20, p.382]{Hart:1977} $$C \sim a'\LH+b'\RH_0, \quad b'>0,\, a'\geq b'e.$$ Expressing the curve using the preferred generators $\LH$ and $\RH$, we write $$C \sim a\LH+b\RH,\quad b>0, \,a\geq 0,$$ and the corresponding curve is said to be of type $(a,b)$.
}

\subsection{The Euler characteristic and the log B--M--Y-inequality}
In this section we establish a result on the topological Euler characteristic of the complement to a curve $C$ on $\Fe$. In this case we view $C$ and $\Fe$ as real manifolds.
\begin{lem}\label{euler}
Let $C$ be a cuspidal curve of genus $g$ and type $(a,b)$ on $\Fe$. Then $$e(\Fe \setminus C)=2g+2.$$
\end{lem}

\begin{proof} 
For the pair $(\Fe,C)$ we have the long exact sequence of cohomology groups

$$
\begin{array}{ccccccc}
0 & \longrightarrow &\H^0(\Fe,C;\Z) &\longrightarrow&\H^0(\Fe;\Z)&\longrightarrow&\H^0(C;\Z)\\
  & \longrightarrow &\H^1(\Fe,C;\Z) &\longrightarrow&\H^1(\Fe;\Z)&\longrightarrow&\H^1(C;\Z)\\
  & \longrightarrow &\H^2(\Fe,C;\Z) &\longrightarrow&\H^2(\Fe;\Z)&\longrightarrow&\H^2(C;\Z)\\
  & \longrightarrow &\H^3(\Fe,C;\Z) &\longrightarrow&\H^3(\Fe;\Z)&\longrightarrow&0\,\\
  & \longrightarrow &\H^4(\Fe,C;\Z) &\longrightarrow&\H^4(\Fe;\Z)&\longrightarrow&0.\\
\end{array}
$$

It is well known that the Hirzebruch surfaces have cohomology groups of the following form,
\begin{align*}
\H^0(\Fe;\Z)&\cong \Z,\\
\H^1(\Fe;\Z)&\cong 0,\\
\H^2(\Fe;\Z)&\cong \Z \oplus \Z,\\
\H^3(\Fe;\Z)&\cong 0,\\
\H^4(\Fe;\Z)&\cong \Z.
\end{align*}

Since a cuspidal curve is homeomorphic to its normalization, we have the following cohomology groups for a cuspidal curve $C$ of genus $g$ (see \cite[Proof of Proposition 1.5.16 pp.42--43]{Fent}),
\begin{align*}
\H^0(C;\Z)&\cong \Z,\\
\H^1(C;\Z)&\cong \Z^{2g},\\
\H^2(C;\Z)&\cong \Z.
\end{align*}

We get the long exact sequence
$$
\begin{array}{ccccccc}
0 & \longrightarrow &\H^0(\Fe,C;\Z) &\longrightarrow&\Z&\longrightarrow&\Z\\
  & \longrightarrow &\H^1(\Fe,C;\Z) &\longrightarrow&0&\longrightarrow&\Z^{2g}\\
  & \longrightarrow &\H^2(\Fe,C;\Z) &\longrightarrow&\Z\oplus \Z&\longrightarrow&\Z\\
  & \longrightarrow &\H^3(\Fe,C;\Z) &\longrightarrow&0&\longrightarrow&0\,\\
  & \longrightarrow &\H^4(\Fe,C;\Z) &\longrightarrow&\Z&\longrightarrow&0.\\
\end{array}
$$

Using Poincaré--Lefschetz duality, we have $\H_i(\Fe \setminus C; \Z) \cong \H^{4-i}(\Fe,C;\Z)$ for $i=0,\ldots,4$. Taking dimensions in the long exact sequence, we find that $e(\Fe \setminus C)=2g+2$.
\end{proof}

Before we state a logarithmic B--M--Y-inequality in this situation, we need to recall some notation and definitions. Note that the logarithmic Kodaira dimension of a non-complete surface $Y$ is denoted by $\ol{\kappa}(Y)$ (see \cite{Fujita, Iitaka}). Now let $V$ be a smooth projective surface, $D$ a reduced SNC-divisor, and $K_V$ the canonical divisor on $V$. If $\ol{\kappa}(V \setminus D)\geq 0$, then there exists a decomposition of $K_V+D$ called the \emph{Zariski--Fujita decomposition} (see \cite[Section 6, pp.527--528]{Fujita}). The decomposition is given by $$K_V+D=H+N,$$ where $H$ and $N$ are $\Q$-divisors, i.e., linear combinations of its prime components with rational coefficients, with the below properties. With $N=\sum n_iN_i$, recall that the intersection matrix $\begin{bmatrix}N_iN_j\end{bmatrix}$ is called negative definite if all its eigenvalues are negative.
\begin{enumerate}[$a)$]
\item $N=0$, or $N$ is an effective $\Q$-divisor with negative definite intersection matrix.
\item $H \ldot C\geq0$ for any effective divisor $C \in \mathrm{Pic}(X)$.
\item $H \ldot N_i=0$ for any prime component $N_i$ of $N$.
\end{enumerate}

The logarithmic B--M--Y-inequality is in \cite{Ore} given in a form that applies to curves on Hirzebruch surfaces, and here we state the inequality as a corollary to this result \cite[Theorem 2.1, p.660]{Ore}.
\begin{cor}\label{cor:BMY}
Let $(V,D)$ be the minimal embedded resolution of a cuspidal curve $C$ of genus $g$ on $\Fe$, and let $K_V$ and $H$ be as in the Zariski--Fujita decomposition.
\begin{enumerate}[$a)$]
\item If $\lkkf\geq 0$, then $$(K_V+D)^2 \leq 3e(\Fe \setminus C)=6g+6.$$ 
\item If $\lkkf=2$, then $$H^2 \leq 3e(\Fe \setminus C)=6g+6.$$
\end{enumerate} 
\end{cor}

\subsection{The logarithmic Kodaira dimension}
In this section we establish a result parallel to a theorem by Wakabayashi in \cite{Wak} concerning the logarithmic Kodaira dimension of complements to curves on Hirzebruch surfaces.

\begin{thm}\label{logkfe} 
On a Hirzebruch surface $\Fe$, let $C$ be an irreducible curve of genus $g$ and type $(a,b)$, with $b>2$ and $a>2-\frac{1}{2}be$, $a>0$.
\begin{itemize}
\item[(I)] If $g>0$, then $\lkkf=2$.
\item[(II)] If $g=0$ and $C$ has at least three cusps, then $\lkkf=2$.
\item[(III)] If $g=0$ and $C$ has at least two cusps, then $\lkkf\geq0$. 
\end{itemize}
\end{thm}

\noi We prove this theorem closely following the proof given by Wakabayashi in \cite{Wak} for the parallel theorem for curves on the projective plane, replacing only the details for $\Po$ with the corresponding details for $\Fe$ where necessary. Note that the proof goes through without essential changes, but that we have different indices in some parts of the proof (cf. \cite{Wak}).

We start by recalling the essential definitions. Let $\LH$ and $\RH$ denote the set of generators of $\mathrm{Pic}(\Fe)$ described above. Let $\sigma:V \longrightarrow \Fe$ be a finite sequence of monoidal transformations, $$V=V_t \xrightarrow{\sigma_t} V_{t-1} \xrightarrow{} \cdots \xrightarrow{} V_1 \xrightarrow{\sigma_1} V_0=\Fe.$$ In short, $$\sigma=\sigma_1 \circ \cdots \circ \sigma_t: V \rightarrow \Fe.$$ Each transformation $\sigma_i$ has exceptional divisor $E_i \subset V_{i}$ and is centered in $p_{i-1} \in V_{i-1}$. Let $E_i^{\prime}$ denote the strict transform of $E_i$ by $\sigma_{i+1}\circ \cdots \circ \sigma_t$. By abuse of notation, we also use the symbol $E_i$ for $(\sigma_{i+1}\circ \cdots \circ \sigma_t)^{*}E_i$, $\RH$ for $\sigma^{*}\RH$ and $\LH$ for $\sigma^{*}\LH$. 

Before giving the proof of the theorem, we need a lemma and a proposition. The formulation and proofs of these are simply adjustments to the ones found in \cite{Wak}.

\begin{lem}\label{lemkfe}
Let $\sigma: V \rightarrow \Fe$, $\RH$, $\LH$ and $E_i$ be as above. For any $\hat{a},\hat{b} \in \mathbb{N},\,n_i \in \mathbb{N} \cup \{0\}$ we have
$$\dim \H^0\Bigl(V,\Os\bigl(\hat{a}\LH+\hat{b}\RH-\sum_{i=1}^tn_iE_i\bigr)\Bigr)\geq \frac{(\hat{b}+1)(2\hat{a}+2+\hat{b}e)}{2}-\sum_{i=1}^{t}\frac{n_i(n_i+1)}{2}.$$
\end{lem}

\begin{proof}
Most of the proof of \cite[Lemma, p.157]{Wak} goes unchanged, since it only concerns local properties of points. A calculation of the dimension of the vector space of polynomials of bigrading $(\hat{a},\hat{b})$ can be found in \cite[Proposition 2.3, p.129]{Laface}.
\end{proof}

Let $C$ be an irreducible curve on $\Fe$ of type $(a,b)$. Let $\sigma:V \rightarrow{} \Fe$ be the minimal embedded resolution of its singularities, such that its reduced inverse image $D$ is an SNC-divisor. Let $C_i$ denote the strict transform of $C$ by $\sigma_{1} \circ \cdots \circ \sigma_i$, and let $m_i$ be the multiplicity of $p_i$ on $C_i$. Let $\tilde{C}$ denote the strict transform of $C$ by $\sigma$. Finally, let $K_V$ denote the canonical divisor on $V$. Then with the sloppy notation introduced above, we have
\begin{eqnarray*}
D&=&\tilde{C}+\sum_{i=1}^tE_i^{\prime},\\
K_V&\sim& (e-2)\LH-2\RH+\sum_{i=1}^tE_i,\\
a\LH+b\RH &\sim& C =\tilde{C}+\sum_{i=1}^t m_{i-1}E_i.
\end{eqnarray*}

\noi Hence, \begin{equation}\displaystyle{D+K_V\sim(a+e-2)\LH+(b-2)\RH+\sum_{i=1}^tE_i^{\prime}-\sum_{i=1}^t(m_{i-1}-1)E_i}.\label{dkfe}\end{equation}

\begin{prp} \label{prpkfe}
With $C$, $a$, $b$, $D$, $K_V$, $\LH$ and $\RH$ as above, suppose that for sufficiently large $k \in \N$ \begin{equation}\lambda k (D+K_V) \sim \lambda \Bigl((a+e-2)\LH+(b-2)\RH \Bigr)+G_k,\label{lineqfe}\end{equation}
where $\lambda$ is a suitable positive number independent of $k$, and $G_k$ is a suitable non-negative divisor on $V$ dependent on $k$. Then $\lkkf=2$.
\end{prp}

\begin{proof}
Choose a $k$ such that (\ref{lineqfe}) holds. Since $G_k$ is non-negative, then for any $n \in \N$ we have $$\dim \H^0\Bigl(V,\Os\bigl(n \lambda k (D+K_V)\bigr)\Bigr)\geq \dim \H^0\Bigl(V, \Os\bigl(n \lambda ((a+e-2)\LH+(b-2)\RH)\bigr)\Bigr).$$ For readability, we will use an even more sloppy notation and write $C+K$ instead of $(a+e-2)\LH+(b-2)\RH$. By Riemann--Roch \cite[Theorem V 1.6, p.362]{Hart:1977}, we have that 

\footnotesize{
\begin{align*}
\h^0\Bigl(n& \lambda \bigl(C+K\bigr)\Bigr)-\h^1\Bigl(n \lambda \bigl(C+K\bigr)\Bigr)+\h^0\Bigl(K_V-n \lambda \bigl(C+K\bigr)\Bigr)\\
&=\frac{1}{2}n \lambda \bigl((a+e-2)\LH+(b-2)\RH\bigr) \ldot \Bigl(n \lambda \bigl((a+e-2)\LH+(b-2)\RH \bigr)-K_V\Bigr)+1+p_a(V).
\end{align*}
}

\noi \normalsize{Rewriting this equation, using that $p_a(V)=p_a(\Fe)=0$ and $\h^1\Bigl(n \lambda (C+K)\Bigr) \geq 0$, we find that}
\footnotesize{
\begin{eqnarray*}
\h^0 \Bigl(n \lambda (C+K)\Bigr)
&=&\frac{1}{2}n \lambda \bigl((a+e-2)\LH+(b-2)\RH\bigr) \ldot \Bigl(n \lambda \bigl((a+e-2)\LH+(b-2)\RH\bigr)-K_V\Bigr)\\
&&+\;1+\h^1 \Bigl(n \lambda (C+K)\Bigr)-\h^0 \Bigl(K_V-n \lambda (C+K)\Bigr)\\
&\geq & \frac{1}{2}n \lambda \bigl((a+e-2)\LH+(b-2)\RH\bigr) \ldot \Bigl(n \lambda \bigl((a+e-2)\LH+(b-2)\RH \bigr)-K_V\Bigr)\\
&&- \;\h^0 \Bigl(K_V-n \lambda (C+K)\Bigr)\\
&\geq & \frac{1}{2}n^2 \lambda^2 (b-2)(2a+be-4)+\frac{1}{2}n \lambda (2a+2b+be-8)\\
&&-\;\h^0 \Bigl(K_V-n \lambda (C+K)\Bigr)\\
&\geq & \frac{1}{2}n^2 \lambda^2 (b-2)(2a+be-4)-\h^0 \Bigl(K_V-n \lambda (C+K)\Bigr).
\end{eqnarray*}}

\normalsize 
\noi Next we show that $$\h^0 \Bigl(K_V-n \lambda ((a+e-2)\LH+(b-2)\RH)\Bigr) \leq 0.$$ Assume to the contrary that on $V$ there exists a positive divisor \begin{equation}P \sim K_V-n \lambda \bigl((a+e-2)\LH+(b-2)\RH\bigr)\label{GLeff}.\end{equation} Then $\sigma(P)$ must be effective. Considering $\Fe$ as a toric variety, the divisor $\LH+\RH$ is in the interior of the nef cone, hence it is ample (see \cite[Example 6.1.16, p.273]{COX}). Therefore, $\sigma(P) \ldot (\LH+\RH)\geq 0$. Because of (\ref{GLeff}) and the conditions on $a$ and $b$, we have that

\begin{eqnarray*}
P \ldot \sigma^{-1}(\LH+\RH)&=&\sigma(P) \ldot (\LH+\RH)\\
&=&\bigl((e-2)\LH-2\RH -n \lambda (a+e-2)\LH -n\lambda(b-2)\RH\bigr) \ldot (\LH+\RH)\\
&=&-e-4-n\lambda \bigl(a+b-4+e(b-1)\bigr)\\
&<&0.
\end{eqnarray*}

\noi This is a contradiction to the above assumption, hence 
$$\h^0 \Bigl(n \lambda \bigl((a+e-2)\LH+(b-2)\RH\bigr)\Bigr) \geq n^2 \lambda^2 \frac{(b-2)(2a+be-4)}{2}.$$
In other words $$\dim \H^0\Bigl(V,\Os\bigl(n \lambda \bigl((a+e-2)\LH+(b-2)\RH \bigr)\bigr)\Bigr) \geq c\cdot n^2$$ for a suitable constant $c>0$ independent of $n$. Note that $c > 0$ because of the conditions on $a$ and $b$. By definition of the logarithmic Kodaira dimension, we then have $\lkkf=2$.
\end{proof}

We now prove Theorem \ref{logkfe} in the same way that Wakabayashi proves the result for curves on $\Po$ in \cite{Wak}. 
\begin{proof}[Proof of Theorem \ref{logkfe}]
{\ \\}
\smallskip
\noi \emph{Case (I)}. Let $C$ be an irreducible curve with $g(C)\geq1$, $b>2$ and $a>2-\frac{1}{2}be$, $a>0$. The genus formula ensures that $$g(C)=\frac{(b-1)(2a-2+be)}{2}-\sum_{i=0}^{t-1}\frac{m_i(m_i-1)}{2}\geq1.$$
With $\hat{a}=a+e-2$, $\hat{b}=b-2$, and $n_i=m_{i-1}-1$ in Lemma \ref{lemkfe}, we get 
$$\dim \H^0\Bigl(V,\Os\bigl((a+e-2)\LH+(b-2)\RH-\sum_{i=1}^t(m_{i-1}-1)E_i\bigr)\Bigr) \geq 1.$$
Hence, the below vector space is non-zero,
$$\H^0\Bigl(V,\Os\bigl((a+e-2)\LH+(b-2)\RH - \sum_{i = 1}^t(m_{i-1}-1)E_i \bigr) \Bigr)\neq 0.$$ Therefore, $$(a+e-2)\LH+(b-2)\RH \sim \sum_{i=1}^{t}(m_{i-1}-1)E_i+G,$$ where $G$ is a positive divisor on $V$. This implies that 
{\small{
\begin{eqnarray*}
k(D+K_V)& \sim &(a+e-2)\LH+(b-2)\RH+(k-1)\Bigl((a+e-2)\LH+(b-2)\RH\Bigr)\\
 &&+ \;k\sum_{i=1}^tE_i^{\prime}-k\sum_{i=1}^t(m_{i-1}-1)E_i\\
&\sim &(a+e-2)\LH+(b-2)\RH +(k-1)G + k\sum_{i=1}^tE_i^{\prime}-\sum_{i=1}^t(m_{i-1}-1)E_i.
\end{eqnarray*}}}\normalsize
\noi Each $E_i$, that is each $(\sigma_{i+1}\circ \cdots \circ \sigma_t)^*E_i$, is a linear combination of the strict transforms $E_j^{\prime}$, $j \geq i$, so for large $k$ the latter three terms in the above sum constitute a non-negative divisor. Hence, we can use Proposition \ref{prpkfe}, with $\lambda=1$, to conclude that $\lkkf=2$.\\
\smallskip

\noi \emph{Case (II)}. Let $C$ be a rational cuspidal curve on $\Fe$. We first assume that $C$ has only one cusp. Let $\qs$ denote the index with the property that $p_{\qs-1}$ is singular on $C_{\qs-1}$ and $p_{\qs}$ is non-singular on $C_{\qs}$. As before, we let $t$ be the number of monoidal transformations such that $D$ is the minimal embedded resolution of $p$ on $C$. We write 
\begin{eqnarray} \label{eienefe}
E_{\qs}&= &E_\qs'+E_{\qs+1}+ \cdots +E_t,\\
E_{t-1}&=&E_{t-1}'+E_t',\notag\\
t-\qs&=&m_{\qs-1}.\notag
\end{eqnarray} 

\noi Using the strategy from \cite{Wak}, we first look at the following vector space,
\footnotesize{\begin{equation}\label{peculiarfe} \displaystyle{\H^0\Bigl(V,\Os\bigl((a+e-2)\LH+(b-2)\RH - \sum_{i \neq \qs}(m_{i-1}-1)E_i-(m_{\qs-1}-2)E_\qs-E_{\qs+1}- \cdots - E_{t-2} \bigr)\Bigr)},\end{equation}}\normalsize and show that this vector space is non-zero. Changing the index in the genus formula gives $$\displaystyle g(C)=\frac{(b-1)(2a-2+be)}{2}-\frac{1}{2}\sum_{i=1}^{t}m_{i-1}(m_{i-1}-1)=0.$$
Rewriting this expression, we have
\footnotesize{
\begin{align*}
\displaystyle \frac{(b-1)(2a-2+be)}{2}-\frac{1}{2}\sum_{i\neq \qs}m_{i-1}(m_{i-1}-1)-\frac{1}{2}(m_{\qs-1}-1)(m_{\qs-1}-2)-(m_{\qs-1}-2)=1.
\end{align*}}
\normalsize
Using Lemma \ref{lemkfe}, we conclude that the vector space in (\ref{peculiarfe}) above is non-zero.

This implies that we may write 
\small{$$\displaystyle (a+e-2)\LH+(b-2)\RH \sim \sum_{i \neq \qs}(m_{i-1}-1)E_i+(m_{\qs-1}-2)E_\qs+E_{\qs+1}+\cdots +E_{t-2}+G_p,$$}\normalsize where $G_p$ is a positive divisor. 

The latter observation can be used together with (\ref{eienefe}) to get an expression for $k(D+K_V)$.
\small{
\begin{flalign*}
\qquad k(D+K_V) \sim&\;  (a+e-2)\LH+(b-2)\RH+(k-1)\bigl((a+e-2)\LH+(b-2)\RH\bigr)\notag&\\
	    & +{}k\sum_{i=1}^tE_i'-k\sum_{i=1}^t(m_{i-1}-1)E_i, &
\end{flalign*}
\begin{flalign*}
\qquad k(D+K_V) \sim & \;(a+e-2)\LH+(b-2)\RH &\\
	     &+{}(k-1)\Bigl(\sum_{i \neq \qs}(m_{i-1}-1)E_i+(m_{\qs-1}-2)E_\qs+E_{\qs+1}+\cdots +E_{t-2}+G_p\Bigr)\notag&\\
 	     &+{}k\sum_{i=1}^tE_i'-k\sum_{i=1}^t(m_{i-1}-1)E_i,&
\end{flalign*}
\begin{flalign*}
\qquad k(D+K_V)\sim & \;(a+e-2)\LH+(b-2)\RH +(k-1)G_p\notag&\\
	     &-{} \sum_{i \neq \qs}(m_{i-1}-1)E_i -k(m_{\qs-1}-1)E_\qs+(k-1)(m_{\qs-1}-2)E_{\qs}\notag&\\
	     &+{}(k-1)\left(E_{\qs+1}+\cdots +E_{t-2}\right)+k\sum_{i=1}^tE_i',&
\end{flalign*}
\begin{flalign*}
\qquad k(D+K_V)\sim & \;(a+e-2)\LH+(b-2)\RH+(k-1)G_p - \sum_{i =1}^t(m_{i-1}-1)E_i &\notag\\
             &-{}(k-1)E_\qs+(k-1)\left(E_{\qs+1}+\cdots +E_{t-2}\right)+k\sum_{i=1}^tE_i',&
\end{flalign*}
\begin{flalign*}
\qquad k(D+K_V) \sim & \;(a+e-2)\LH+(b-2)\RH+(k-1)G_p - \sum_{i =1}^t(m_{i-1}-1)E_i &\notag\\
	     &+(k-1)\left(-E_\qs+E_{\qs+1}+\cdots +E_{t-2}\right)+k\sum_{i=1}^tE_i',&
\end{flalign*}
\begin{flalign*}
\qquad k(D+K_V) \sim & \;(a+e-2)\LH+(b-2)\RH+(k-1)G_p - \sum_{i =1}^t(m_{i-1}-1)E_i&\notag\\
 &+(k-1)\left(-E_\qs'-E_{t-1}-E_t\right) +k\sum_{i=1}^tE_i',&
\end{flalign*}
\begin{flalign*}
\qquad k(D+K_V) \sim & \;(a+e-2)\LH+(b-2)\RH+(k-1)G_p \notag\\
&- \sum_{i =1}^t(m_{i-1}-1)E_i +k\sum_{i=1}^tE_i'-(k-1)\left(E_\qs'+E_{t-1}'+2E_t'\right)\label{kDKfe}&
\end{flalign*}
}\normalsize

Then we make the assumption that $C$ has three cusps, $p_1$, $p_2$ and $p_3$. Note that the following procedure also works if we assume that $C$ has more than three cusps. We perform successive minimal embedded resolutions of the cusps, and take one cusp at the time until we reach $V$. Let $\hat{t}_j$ denote the number of monoidal transformations needed to resolve the cusps $p_1, \ldots, p_j$, but not $p_{j+1}, \ldots$, $j=1,2,3$. To resolve the three singularities in such a way that $D$ is an SNC-divisor, we must apply in total $t:=\hat{t}_3$ successive monoidal transformations to the curve. We let $\hat{\qs}_j$ denote the smallest index such that the cusps $p_1,\ldots,p_{j-1}$ are resolved and that in the process of resolving $p_j$, the curve $C_{\hat{\qs}_j-1}$ is singular at $p_{\hat{\qs}_j-1}$, but $C_{\hat{\qs}_j}$ is non-singular at $p_{\hat{\qs}_j}$. For each cusp $p_j$ we have that $\hat{t}_j=\hat{\qs}_j+m_{\hat{\qs}_j-1}$.

The minimal embedded resolution of the curve can be viewed in three different ways, and we use this to find three positive divisors $G_{p,j}$ and similar expressions to the above for $k(D+K_V)$ on the surface $V$. Note that we now sum up to $t$. For each $j$ we may write
\small{
\begin{align*}k(D+K_V)\sim& \;(a+e-2)\LH+(b-2)\RH\\
&+(k-1)G_{p,j}- \sum_{i =1}^t(m_{i-1}-1)E_i\\ 
&+k\sum_{i=1}^tE_i'-(k-1)\left(E_{\hat{\qs}_j}'+E_{\hat{t}_j-1}'+2E_{\hat{t}_j}'\right).\end{align*}
}\normalsize

\noi We then add the three expressions and get
\small{
\begin{align*}3k(D+K_V)\sim&\;3\bigl((a+e-2)\LH+(b-2)\RH\bigr)\\
&+(k-1)\sum_{j=1}^3G_{p,j} - 3\sum_{i=1}^t(m_{i-1}-1)E_i \\
&+3k\sum_{i=1}^tE_i'-(k-1)\sum_{j=1}^3\left(E_{\hat{\qs}_j}'+E_{\hat{t}_j-1}'+2E_{\hat{t}_j}'\right).
\end{align*}}\normalsize

\noi The latter two lines of the sum constitutes a non-negative divisor for large $k$. The conclusion then follows by Proposition \ref{prpkfe}, and we have $\lkkf=2$.\\
\smallskip

\noi \emph{Case (III)}.
If $C$ has two cusps $p_1$ and $p_2$, then as in Case (II) we can look at each cusp separately and find two expressions on the form
$$\displaystyle (a+e-2)\LH+(b-2)\RH \sim \sum_{i \neq \qs}(m_{i-1}-1)E_i+(m_{\qs-1}-2)E_{\qs}+E_{{\qs}+1}+\cdots +E_{{t}-2}+G_{p,j},$$ where $G_{p,j}$ is a positive divisor for each $j=1,2$. 

By performing the blowing-ups of the cusps successively, with the same indices as in Case (II), we can use (\ref{dkfe}),
$$D+K_V \sim \sum_{i=1}^t E'_i+G_{p,j}-E_{\hat{\qs}_j}+E_{\hat{\qs}_j+1}+\cdots + E_{{\hat{t}_j}-2}.$$ Summing these expressions, we get
$$2(D+K_V) \sim 2\sum_{i=1}^t E_i'+\sum_{j=1}^2G_{p,j}+\sum_{j=1}^2(-E_{\hat{\qs}_j}+E_{\hat{\qs}_j+1}+\cdots + E_{{\hat{t}_j}-2}).$$ Using (\ref{eienefe}), we then get
$$2(D+K_V) \sim 2\sum_{i=1}^t E_i'+\sum_{j=1}^2G_{p,j}-\sum_{j=1}^2(E'_{\hat{\qs}_j}+E'_{\hat{t}_j-1}+2E'_{\hat{t}_j}).$$
The right hand side is a positive divisor, hence $$\H^0\Bigl(V,\Os\bigl(2(D+K_V)\bigr)\Bigr) \neq 0.$$ It follows that $ \lkkf \geq 0$.
\end{proof}

\section{On the number of cusps}\label{ON}
In this section we find an upper bound for the number of cusps on a cuspidal curve on a Hirzebruch surface. This result is a modification of Theorem \ref{ONCP2} and its proof by Tono \cite{Tono05}, and essentially everything in the proof goes unchanged. We include the proof here for the sake of completion.

We first recall a few preliminary definitions and results needed in the proof. Let $D$ be a reduced effective SNC-divisor on a nonsingular projective surface $V$. We write $D$ as the sum of its irreducible components $D_i$, that is, $D=D_1+\ldots+ D_r$. 

We have from \cite{Fujita, Tsunoda, Tono05} a number of important notions related to $D$. We define the \emph{branching number} of $D_i$, $\beta(D_i)=(D - D_i)\ldot D_i$. The component $D_i$ is called an \emph{isolated component} of $D$ if $\beta(D_i)=0$. If $\beta(D_i)=1$, then $D_i$ is called a \emph{tip}. If $\beta(D_i)\geq 3$, then $D_i$ is called a \emph{branching component} of $D$. A partial sum of components of $D$, say $L=D_1+\ldots+D_m$, is called a \emph{linear chain} of $D$ if $\beta(D_1)=1$, $\beta(D_i)=2$ for $2 \leq i \leq m-1$, and $D_i \ldot D_{i+1}=1$ for $1\leq i \leq m-1$. If $\beta(D_m)=1$, then $L$ is called a \emph{rod}. If $\beta(D_m)=2$, then $L$ is called a \emph{twig}. In the latter case, $L$ is connected to $D$ by a component $D_{m+1} \notin L$. If $\beta (D_{m+1}) \geq 3$, that is $D_{m+1}$ is a branching component, then $L$ is called a \emph{maximal twig}. A linear chain is called \emph{rational} if $D_i$ is a rational curve for every $i$. It is called \emph{admissible} if $D_i^2 \leq -2$ for every $i$. 


A divisor on $V$ is called \emph{contractible} if the intersection matrix of its irreducible components is negative definite. If a linear chain $L$ of $D$ is rational and admissible, then it is contractible \cite{Tono05}. Moreover, there exists a unique \Q-divisor ${\mathrm{Bk}}(L)$, called the \emph{bark} of $L$, with the property that $(K+D)\ldot D_i={\mathrm{Bk}}(L) \ldot D_i$ for every $i$. 

A component $F$ of $D$ consisting of three rational admissible maximal twigs and a rational curve $F_1$ is called a \emph{fork} if $(K+F+B) \ldot F_1<0$, where $B$ is the sum of the barks of the three maximal twigs. A fork is called admissible if $F_1^2\leq -2$, and a fork is admissible if and only if it is contractible \cite{Tono05}.

The \emph{bark} of $D$, ${\mathrm{Bk}}(D)$ is defined to be the sum of the barks of all rational admissible rods, rational admissible forks and the remaining rational admissible twigs.

We call the the pair $(V,D)$ \emph{almost minimal} if for every irreducible curve $M$ in $V$, either $(K+D-{\mathrm{Bk}}(D))\ldot M\geq 0$ or $(K+D-{\mathrm{Bk}}(D))\ldot E<0$ and $\mathrm{Bk}(D)+M$ is \emph{not} contractible.

We also need the following proposition before we state and prove the main theorem. This proposition holds for nonsingular projective surfaces defined over $\C$, and it is proved by Tono in \cite[Corollary 4.4, p.219]{Tono05}, here stated for our situation.
\begin{prp}\label{Tono44}
Let $l$ denote the number of rational maximal twigs of $D$. If $\ol{\kappa}(V \setminus D)=2$, if the pair $(V,D)$ is almost minimal, and if $D$ contains neither a rod consisting of $(-2)$-curves nor a fork consisting of $(-2)$-curves, then
\begin{equation}
l\leq 12e(V\setminus D) + 5 - 3 p_a(D).
\end{equation}
\end{prp}

We are now ready to give an upper bound on the number of cusps on a rational cuspidal curve on $\Fe$. The result is similar to the one given by Tono in \cite{Tono05} for $\Po$ (cf. Theorem \ref{ONCP2}).

\begin{thm}\label{ONCfe}
The number of cusps $s$ on a cuspidal curve $C$ of genus $g$ on a Hirzebruch surface $\Fe$ has an upper bound, $$s \leq \frac{21g+29}{2}.$$
\end{thm}

\begin{proof}[Proof (cf. Tono \cite{Tono05})] The proof given by Tono in \cite{Tono05} for $\Po$ is directly applicable in this situation, and the following is essentially the same proof. We include the proof here for the sake of completion, and at some places we have chosen to write out the details more carefully than in the original proof. 

The aim of the proof is to set up a situation where we can apply Proposition \ref{Tono44}, and then the theorem follows. 

We now construct the surface to which we can apply Proposition \ref{Tono44}, and first show that two of the prerequisites in the proposition hold for this surface. Let $C=\V(F)$ be a cuspidal curve of genus $g$ on $\Fe$. Let $s$ denote the number of cusps on $C$. We are looking for an upper bound of $s$, hence we may assume that $s\geq 3$. 

Let $\sigma:V \longrightarrow \Fe$ denote the minimal embedded resolution of $C$, and let $D=\tilde{C}+\sum_{i=1}^tE_i$ be the reduced total inverse image of $C$ on the surface $V$. For a cusp $p$, the dual graph of $\sigma^{-1}(p)+\tilde{C}$ has the shape given in Figure \ref{mercusp}.

\begin{figure}[ht] 
  \begin{center}
\includegraphics[trim = 45mm 230mm 45mm 30mm, clip, width=1\linewidth]{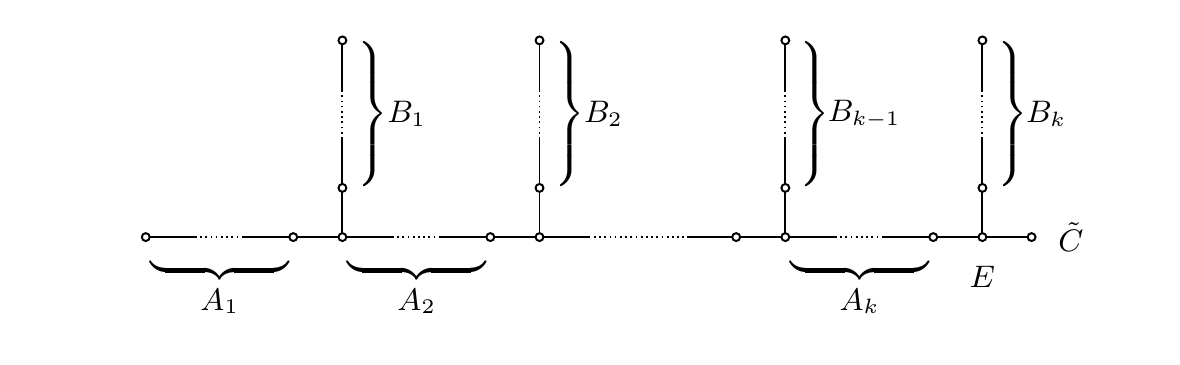}
  \end{center}
   \caption[The dual graph of the minimal embedded resolution of a cusp.]{The dual graph of $\sigma^{-1}(p)+\tilde{C}$.}
\label{mercusp}
\end{figure}

In Figure \ref{mercusp}, $E$ denotes the last blowing up in the resolution of $p$, and we have $E^2 =-1$. All other curves in $D$ have self intersection $\leq -2$. Notice that the morphism $\sigma$ can be viewed as successive contractions in a way that can be handled with a quite clean notation. For a cusp $p$, $\sigma$ first contracts $E+A_k+B_k$ in Figure \ref{mercusp} to a $(-1)$-curve $E'$. The process then continues in the same manner, with the contraction of $E'+A_{k-1}+B_{k-1}$ to another $(-1)$-curve and so on, until we reach $\Fe$. 

Considering the graph of the minimal embedded resolution of all cusps on $C$, we see that $D$ is connected. Notice that $D$ contains $s$ curves $E_j$ with self intersection $-1$, all of which are branching components, and that the strict transform $\tilde{C}$ of $C$ is also a branching component when $s\geq3$.

We do not know if the pair $(V,D)$ is almost minimal, so we cannot use Proposition \ref{Tono44} on this surface directly. We solve this problem by applying a theorem by Tsunoda in \cite{Tsunoda} (see \cite[Lemma 3.2, p.218]{Tono05}). By \cite[Theorem 1.11, p.226]{Tsunoda}, there exists a birational morphism $\mu:V \longrightarrow V'$, consisting of successive contractions of $(-1)$-curves such that, with $D'=\mu_{*}D$, the pair $(V',D')$ is almost minimal and $\ol{\kappa}(V \setminus D)=\ol{\kappa}(V' \setminus D')$. Since we assume that $s\geq 3$, by definition, Theorem \ref{logkfe}, and \cite[Theorem 1.11, p.226]{Tsunoda}, we have $$2=\lkkf=\ol{\kappa}(V \setminus D)=\ol{\kappa}(V' \setminus D').$$

Before we show that the third prerequisite in Proposition \ref{Tono44} holds for $(V',D')$, we must estimate some of the invariants involved in the formula in Proposition \ref{Tono44}. We begin with the Euler characteristic. By Lemma \ref{euler} we have $e(V\setminus D)=e(\Fe \setminus C)=2g+2$. To determine $e(v' \setminus D')$ we investigate the morphism $\mu$ more closely. The morphism $\mu$ is a composition of contractions, and we let $M_1,\ldots, M_n \subset V$ denote the strict transforms of the $(-1)$-curves that are contracted by $\mu$ and \emph{not} contained in $D$. Observe that $\mu$ possibly contracts some $(-1)$-curves contained in $D$ in addition to the $M_j$'s, but these contractions would not affect the Euler characteristic of the complement. Since $D$ is connected, we must by Tono \cite[Lemma 3.4, p.218]{Tono05} have $D\ldot M_j\leq 1.$ Moreover, since $V \setminus D \cong \Fe \setminus C \cong \D_{+}(F)$, it is affine. Hence, $M_j \cong \Pot$ cannot be contained in $V\setminus D$, and therefore $D \ldot M_j=1\label{TONO:affin n}$. Using this information, we calculate the Euler characteristic, \begin{equation*}e(V'\setminus D')=e(V \setminus D)-n=2g+2-n. \label{TONO:EC}\end{equation*} Note that since $V' \setminus D' \cong V \setminus D$ is affine, it follows that $e(V' \setminus D')>0$. 

Our next aim is to ensure that the number $l$ of rational maximal twigs of $D'$ can be estimated by the number of cusps on $C$. This estimate relies on the fact that some of the components of $D$ cannot be contracted by $\mu$. 

For each cusp $p_j$, $j=1, \ldots, s$, let $E_j$ denote the $(-1)$-curve that intersects $\tilde{C}$ in the minimal embedded resolution. We will now show by contradiction that $\mu$ does not contract any $E_j$. Assume for contradiction that one $E_j$, say $E$, is contracted by $\mu$. Now $E$ is a branching component of $D$, hence it cannot be directly contracted by $\mu$ \cite[Lemma 3.4, p.218]{Tono05}, and $\mu$ must contract a $(-1)$-curve that intersects $E$. Contracting the $(-1)$-curve intersecting $E$ turns $E$ into a curve with nonnegative self intersection. Then $E$ cannot be contracted, contrary to the assumption. We conclude that $E_j$ cannot be contracted for any cusp. 

We additionally have to ensure that $\tilde{C}$ is not contracted by $\mu$. This can be shown by induction on the number of blowing downs in the morphism $\mu$. Note that this part of the proof is also by Tono (personal communication). Let $\mu=\mu_{\nu} \circ \cdots \circ \mu_1$, ${\nu}\geq n$, be a decomposition of $\mu$. Then $\mu_1$ cannot contract $\tilde{C}$ since $s \geq3$ makes $\tilde{C}$ a branching component of $D$, that is, $\beta_D(\tilde{C})\geq3$. That would contradict \cite[Lemma 3.4, p.218]{Tono05}. So suppose that $\mu_k \circ \cdots \circ \mu_1$, $k<n$, does not contract the strict transform of $\tilde{C}$ by $\mu_{k-1}\circ \cdots \circ \mu_1$. Let $D^k$ and $\tilde{C}^k$ denote the strict transforms of $D$ and $\tilde{C}$ under $\mu_k \circ \cdots \circ \mu_1$. Now since $\mu$ does not contract any of the last exceptional curves $E_j$ for any cusp $p_j$, $\tilde{C}^k$ will still be a branching component of $D^k$. Then $\mu_{k+1}$ cannot contract $\tilde{C}^k$, because that would contradict \cite[Lemma 3.4, p.218]{Tono05}. Hence, $\mu_{k+1}\circ \cdots \circ \mu_1$ does not contract $\tilde{C}$. So by induction, $\tilde{C}$ cannot be contracted by $\mu$. 

The number $l$ of rational maximal twigs of $D'$ can now be estimated by the number of cusps on $C$. For each cusp $p$, let $A=\sigma^{-1}(p)-E-B_k$. The morphism $\mu$ affects the tree of rational curves $A+E+B_k$, and contracts it at most to another tree of rational curves. Since $E$ is not contracted by $\mu$, $\mu(E)$ must be a curve with self intersection $\geq -1$. Then by \cite[Lemma 3.5, p.218]{Tono05}, $\mu(E)$ cannot be part of any rational linear chain or fork, hence not part of any rational maximal twig of $D'$. This implies that $A$ cannot be contracted to a point by $\mu$. Furthermore, $\mu(B_k)$ can be contracted to a point, but then $\mu(A)$ has to contain at least two rational maximal twigs in order to avoid that $\mu(E)$ is part of a rational maximal twig. Summing up, we observe that $D'$ must have at least two rational maximal twigs per cusp, so we have $2s \leq l$.

Now we note that the third prerequisite in Proposition \ref{Tono44} holds for $(V',D')$. The morphism $\mu$ does not disconnect $D$, so $D'$ is connected. Since $D'$ is connected and additionally has at least $6 \leq 2s$ maximal twigs, it is impossible that it contains a rod consisting of $(-2)$-curves or a fork consisting of $(-2)$-curves.

Proposition \ref{Tono44} additionally involves the invariant $p_a(D')$, which is equal to $g$ in this case. Indeed, since $\tilde{C}$ is nonsingular, $p_a(\tilde{C})=g$. Since $D'=\tilde{C}+\sum E_i$, not necessarily for all $i$, and since $D'$ is an $SNC$-divisor, we have that $p_a(D')=p_a(\tilde{C})=g$.

By Proposition \ref{Tono44} applied to $(V', D')$ and the above estimates, we then find the desired upper bound on the number of cusps,
\begin{align*}
2s &\leq 12(2g+2-n)+5-3g \notag\\
 & \leq 21g+29-12n\\
 &\leq 21g+29.
\end{align*} 

\end{proof}

We immediately get a corollary for rational cuspidal curves on $\Fe$ (see \cite{MOECCH}).
\begin{cor}
A rational cuspidal curve on $\Fe$ can not have more than 14 cusps.
\end{cor}

Note that in the subsequent article \cite{MOECCH} we find examples of rational cuspidal curves with four cusps on Hirzerbruch surfaces. Although the bound that we have found here is 14, there are, however, no examples of curves with more than four cusps. The lack of examples leads us to the following conjecture.
\begin{conj*}
A rational cuspidal curve on a Hirzebruch surface has at most four cusps.
\end{conj*} 

\bibliographystyle{hacm}
\bibliography{bib2}

\end{document}